\documentclass[12pt]{amsart}
\usepackage{amssymb,verbatim,enumerate,ifthen}
\usepackage{cite}
\usepackage[mathscr]{eucal}
\usepackage[utf8]{inputenc}
\usepackage[T1]{fontenc}
\usepackage{esint}
\usepackage{marginnote}
\newtheorem{thm}{Theorem}[section]
\newtheorem{cor}[thm]{Corollary}
\newtheorem{prop}[thm]{Proposition}
\newtheorem{lem}[thm]{Lemma}

\theoremstyle{definition}

\newtheorem{exa}[thm]{Example}

\numberwithin{equation}{section}
\def\eq#1{{\rm(\ref{#1})}}
\def\Eq#1#2{\ifthenelse{\equal{#1}{*}}
  {\begin{equation*}\begin{aligned}[]#2\end{aligned}\end{equation*}}
  {\begin{equation}\begin{aligned}[]\label{#1}#2\end{aligned}\end{equation}}}

\newcommand\case[1]{\noindent{\sc Case #1.}}

\def\G{\mathscr{G}}

\def\calL{\mathcal{L}}

\newcommand\R{\mathbb{R}}
\newcommand\N{\mathbb{N}}

\newcommand{\abs}[1]{\left| #1 \right| }

\DeclareMathOperator{\ord}{ord}
\DeclareMathOperator{\uo}{uord}
\DeclareMathOperator{\lo}{lord}
\DeclareMathOperator{\sign}{sign}

\newcommand{\dotvec}[3][SKIPPED]{
\ifthenelse{\equal{#1}{SKIPPED}}
  {#2,\dots,#3}
  {\underbrace{#2,\dots,#3}_{#1\text{ entries}}}
}

\newcommand{\Hc}[2][SKIPPED]{
\ifthenelse{\equal{#1}{SKIPPED}}
  {
    \ifthenelse{\equal{#2}{}}
      {\mathscr{H}}
      {\mathscr{H}(#2)}
  }
  {
    \ifthenelse{\equal{#2}{}}
      {\mathscr{H}_{#1}}
      {\mathscr{H}_{#1}(#2)}
  }
}
\newcommand{\Est}[2][SKIPPED]{
\ifthenelse{\equal{#1}{SKIPPED}}
  {
    \ifthenelse{\equal{#2}{}}
      {\mathscr{C}}
      {\mathscr{C}(#2)}
  }
  {
    \ifthenelse{\equal{#2}{}}
      {\mathscr{C}_{#1}}
      {\mathscr{C}_{#1}(#2)}
  }
}

\newcommand\PG{\mathcal{PG}}
\newcommand\GPG{\mathcal{GPG}}
\newcommand\HS{\mathcal{HS}}

\begin{document}
\begin{abstract}
    We generalize the result of (Witkowski,~2014) which binds orders of homogeneous, symmetric means $M,N,K \colon\mathbb{R}_+^2 \to \mathbb{R}_+$ of power growth that satisfy the invariance equation 
  $K(M(x,y),N(x,y))=K(x,y)$ to the broader class of means.
  
  Moreover, we define the lower-- and the upper--order which gives us insight into the order of the solution of this equation in the case when means do not belong to this class.
\end{abstract}
\title{On the invariance equation for means of generalized power growth}
\author[P. Pasteczka]{Pawe\l{} Pasteczka}
\address{Institute of Mathematics \\  University of the National Education Commission \\ Podchor\k{a}\.zych str. 2, 30-084 Krak\'ow, Poland}
\email{pawel.pasteczka@up.krakow.pl}

\subjclass[2020]{26E60}
\keywords{Invariant means, monomial, power means, Gini means, power growth}
\maketitle

\section{Introduction}
The result by Witkowski \cite{Wit14} allows one to establish the order of the solution of the invariance functional equation in the case of means of power growth. We will generalize it in a few ways. First, we introduce a notion of the lower-- and the upper--order. Then we call the mean to be of generalized power growth if these orders are equal. It turns out that this definition generalizes the one introduced in \cite{Wit14}. Having this established, we reprove the main result of the mentioned paper in this more general setup. At the very end, we calculate the generalized order of several nontrivial means. A few of them are of generalized power growth but not of power growth.

\subsubsection*{Invariance property}
The study of invariant means is a classical problem of the iteration theory which has its origin in the Gauss' study \cite{Gau18}. Let us first recall that a \emph{(bivariate) mean} (on $\R_+:=(0,\infty)$) is an arbitrary function $M \colon \R_+\times \R_+ \to \R_+$ satisfying the so-called \emph{mean property}, that is
\Eq{*}{
\min(x,y)\le M(x,y)\le \max(x,y)\quad \text{ for all }x,y \in \R_+.
}
In the most classical setup for a given means $M,\ N \colon \R_+\times \R_+ \to \R_+$ we are searching for a mean $K \colon \R_+\times \R_+ \to \R_+$ such that
\Eq{*}{
K\big(M(x,y),N(x,y)\big)=K(x,y)\quad\text{for all }x,y \in \R_+.
}
Then we say that $K$ is \emph{$(M,N)$-invariant}. Under the assumption that both $M, N$ are continuous and 
\Eq{*}{
\abs{M(x,y)-N(x,y)}<\abs{x-y} \text{ for all }x,y\in \R_+\text{ with }x \ne y,
}
there exists exactly one $(M,N)$-invariant mean which is also continuous; see, for example, \cite[Theorem 8.2]{BorBor87}. For details regarding the problem of invariance as well as related issues, we refer the reader to the recent paper Jarczyk--Jarczyk \cite{JarJar18}.

\subsubsection*{Properties of means, means of power growth}
It is quite usual to assume that means are symmetric (that is $M(x,y)=M(y,x)$ for all $x,y \in \R_+$) and homogeneous (which states that $M(tx,ty)=t M(x,y)$ for all $x,y,t \in \R_+$). There is a folk result that states that whenever $M, N$ are both symmetric (resp. homogeneous) then the $(M, N)$-invariant mean also admits these properties (provided that it is uniquely determined). Thus from now on let $\HS$ denote the family of all homogeneous, symmetric means $M \colon \R_+^2 \to \R_+$.

Following the idea contained in \cite{Wit14}, we say that a mean $M \in \HS$ is of \emph{power growth}
if there exist a real number $\ord(M)$ and a number $C_M \in (0,\infty)$ such that
\Eq{*}{
\lim_{x \to 0^+} \frac{M(x,1)}{x^{\ord(M)}}=C_M.
}
We shall call $\ord(M)$ the \emph{order of $M$}.
The class of all homogeneous, symmetric means $M \colon \R_+^2 \to \R_+$ of power growth will be denoted by $\PG$. Let us recall the result by Witkowski~\cite{Wit14}.

\begin{prop}[\!\!\cite{Wit14}, Theorem 1]
 Let $M,\,N,\,K \in \PG$. Assume $\ord(M) \ge \ord(N)$ and
\Eq{*}{
C_M^{\ord(K)} C_N^{1-\ord(K)} \ne 1 \qquad \text{or} \qquad \ord(K)(1 - \ord(M) + \ord(N)) \ne \ord(N).
}
If $K(M(x, y), N(x, y)) = K(x, y)$, then $\ord(K) = \ord(M) = \ord(N)$.
\end{prop}

This result has an immediate corollary.
\begin{cor}\label{cor:PG}
Let $M,\,N,\,K \in \PG$ with $\ord(M) \ge \ord(N)$. 
  If $K$ is $(M,N)$-invariant then 
  \Eq{*}{
  \ord(K)(1 - \ord(M) + \ord(N)) = \ord(N).
  }
\end{cor}

The aim of this paper is to reprove Corollary~\ref{cor:PG} for a wider family of means, which we call the means of generalized power growth. Moreover, we study two side problems, that is: \\ --- to establish certain bounds for the order of invariant mean using the idea of lower and upper orders,\\--- to prove that there exist certain means which are of generalized power growth but not of power growth. 

\section{Lower and upper orders}

In this section, we introduce the notion of lower and upper orders. First observe that, roughly speaking, the mean $M \in \HS$ is of power growth $\alpha$ if (and only if) there exists $C>0$ such that $M(x,1)=Cx^\alpha(1+o(1))$ for $x$ close to $0$.
In this spirit, $M \in \HS$ is going to be of generalized power growth $\alpha$ if $M(x,1)=x^{\alpha+o(1)}$ for $x$ close to $0$. In order to make this possible, we introduce the lower and the upper order first.

For $M \in \HS$ we define the \emph{lower order} $\lo(M)$ and the \emph{upper order} $\uo(M)$ by 
\Eq{*}{
\lo(M)&:=\sup\Big\{\alpha \in \R \colon \limsup_{x \to 0^+}\frac{M(x,1)}{x^\alpha}=0 \Big\},\\
\uo(M)&:=\inf\Big\{\alpha \in \R \colon \liminf_{x \to 0^+}\frac{M(x,1)}{x^\alpha}=\infty \Big\}.
}

Clearly, since $\frac{M(x,1)}{ x^\alpha}$ is positive for all $x>0$, we have
\Eq{*}{
\limsup_{x \to 0^+}\frac{M(x,1)}{x^\alpha}=0\phantom{\infty} &\Longrightarrow \:\:\frac{M(x,1)}{x^\alpha}\text{ is convergent to }0,\\
\liminf_{x \to 0^+}\frac{M(x,1)}{x^\alpha}=\infty\phantom{0} &\Longrightarrow \:\:\frac{M(x,1)}{x^\alpha}\text{ is divergent.}
}
On the other hand we cannot assume anything related to convergence (divergence) of the function $x\mapsto\frac{M(x,1)}{x^\alpha}$ in the definition of lower and upper orders. Now we prove that both $\lo(M)$ and $\uo(M)$ are finite for every symmetric and homogeneous mean $M$. 

\begin{lem}\label{lem:louocomp}
Let $M \colon \R_+^2 \to \R_+$ be a symmetric and homogeneous mean. Then 
\begin{align}
\liminf_{x \to 0^+}\frac{M(x,1)}{x^\alpha}&=\infty\phantom{0} \qquad \text{ for all } \alpha>\uo(M), \label{X1}\\
\limsup_{x \to 0^+}\frac{M(x,1)}{x^\alpha}&=0\phantom{\infty} \qquad \text{ for all } \alpha<\lo(M). \label{X2}
\end{align}
Moreover $0\le \lo(M)\le \uo(M)\le 1$.
\end{lem}
\begin{proof}
 Fix $\alpha>\uo(M)$. By the definition there exists $\beta \in (\uo(M),\alpha)$ such that 
 \Eq{*}{
\liminf_{x \to 0^+}\frac{M(x,1)}{x^\beta}=\infty.
}
 But $x^\alpha<x^\beta$ for all $x \in (0,1)$, thus 
 \Eq{*}{
\liminf_{x \to 0^+}\frac{M(x,1)}{x^\alpha} \ge \liminf_{x \to 0^+}\frac{M(x,1)}{x^\beta} =\infty,
}
which is \eq{X1}. This inequality also implies 
 \Eq{*}{
\limsup_{x \to 0^+}\frac{M(x,1)}{x^\alpha}=\liminf_{x \to 0^+}\frac{M(x,1)}{x^\alpha} =\infty\qquad \text{ for all }\alpha > \uo(M),
}
which yields $\lo(M)\le \uo(M)$. 
To prove the inequality $\uo(M) \le 1$, it is sufficient to observe that 
\Eq{*}{
\liminf_{x \to 0^+}\frac{M(x,1)}{x^\alpha}\ge \lim_{x \to 0^+} x^{1-\alpha}=\infty \qquad \text{for every }\alpha>1.
}
Analogously to \eq{X1}
we can show \eq{X2}. Moreover
\Eq{*}{
\limsup_{x \to 0^+}\frac{M(x,1)}{x^\alpha}\le \lim_{x \to 0^+} x^{-\alpha}=0 \qquad\text{ for all }\alpha<0,
}
and thus $\lo(M)\ge 0$. 
\end{proof}

In the next proposition, we show that we can use these values to describe the behavior of the mapping $y\mapsto M(y,1)$ close to infinity. In this result, we introduce the equivalent definition of both orders and prove the related counterpart of the above lemma. %Lemma~\ref{lem:louocomp}.

\begin{prop}
    Let $M \in \HS$. Then
\Eq{*}{
\inf\Big\{\beta \in \R \colon \limsup_{y \to \infty}\frac{M(y,1)}{y^\beta}=0 \Big\}&=1-\lo(M),\\
\sup\Big\{\beta \in \R \colon \liminf_{y \to \infty}\frac{M(y,1)}{y^\beta}=\infty \Big\}&=1-\uo(M).
}
Furthermore
\Eq{*}{
\limsup_{y \to \infty}\frac{M(y,1)}{y^\beta}&=0\phantom{\infty}\qquad \text{ for all } \beta>1-\lo(M),\\
\liminf_{y \to \infty}\frac{M(y,1)}{y^\beta}&=\infty\phantom{0}\qquad \text{ for all } \beta<1-\uo(M).\\
}
\end{prop}
\begin{proof}
    Since $M$ is homogeneous, we have
    \Eq{*}{
    \limsup_{y \to \infty}\frac{M(y,1)}{y^\beta}=\limsup_{y \to \infty}\frac{y M(1,\frac1y)}{y^\beta}=\limsup_{y \to \infty}\frac{M(1,\frac1y)}{y^{\beta-1}}.
    }
    Thus, putting $x:=\frac1y$, we obtain
    \Eq{221}{
    \limsup_{y \to \infty}\frac{M(y,1)}{y^\beta}=\limsup_{x \to 0^+}\frac{M(1,x)}{x^{1-\beta}}.
    }
Consequently,
\Eq{*}{
\inf\Big\{\beta \in \R \colon \limsup_{y \to \infty}\frac{M(y,1)}{y^\beta}=0 \Big\}
&=\inf\Big\{\beta \in \R \colon \limsup_{x \to 0^+}\frac{M(1,x)}{x^{1-\beta}}=0 \Big\}\\
&=\inf\Big\{1-\alpha \colon \alpha \in \R,\ \limsup_{x \to 0^+}\frac{M(1,x)}{x^{1-(1-\alpha)}}=0 \Big\}\\
&=1-\sup\Big\{\alpha \in \R \colon \limsup_{x \to 0^+}\frac{M(1,x)}{x^{\alpha}}=0 \Big\}\\
&=1-\lo(M).
}
Similarly, we can show that \eq{221} holds for the lower limits, whence 
\Eq{*}{
\sup\Big\{\beta \in \R \colon \liminf_{y \to \infty}\frac{M(y,1)}{y^\beta}=\infty \Big\}&=\sup\Big\{\beta \in \R \colon \liminf_{x \to 0^+}\frac{M(1,x)}{x^{1-\beta}}=\infty \Big\}\\
&=1-\inf\Big\{\alpha \in \R \colon \liminf_{x \to 0^+}\frac{M(1,x)}{x^{\alpha}}=\infty \Big\}\\
&=1-\uo(M).
}
The rest of the proof is analogous to the proof of Lemma~\ref{lem:louocomp}, and therefore omitted.
\end{proof}

Finally, we show that both (lower- and upper-) orders are decreasing with respect to the standard ordering of means.
\begin{prop}\label{Prop33}
     Let $M,N \in \HS$. If there exists $c \in \R$ such that $M(x,y) \le c N(x,y)$ for all $x,y \in \R_+$. Then $\lo(M)\ge \lo(N)$ and $\uo(M)\ge \uo(N)$.
\end{prop}
\begin{proof}
    In view of Lemma~\ref{lem:louocomp}, for all $\alpha < \lo(N)$  we have
    \Eq{*}{
    0 \le \limsup_{x \to 0^+}\frac{M(x,1)}{x^\alpha} \le \limsup_{x \to 0^+}\frac{cN(x,1)}{x^\alpha}=0,
    }
    whence, by the definition of $\lo$, we get $\lo(M) \ge \lo(N)$. The proof of the second inequality is analogous.
\end{proof}

\section{Means of generalized power growth}
Motivated by Lemma~\ref{lem:louocomp}, mean $M\in\HS$ is said to be of \emph{generalized power growth} provided the equality $\lo(M)=\uo(M)$ holds. The
class of all homogeneous means of generalized power growth will be denoted by $\GPG$. For every mean $M \in \GPG$ we define the \emph{generalized order} (we refer to it simply as the \emph{order}) by $\ord^*(M):=\lo(M)$ or, equivalently, $\ord^*(M)=\uo(M)$.

We show that every mean which is of power growth is also of generalized power growth (that is $\PG \subset \GPG$ ) and the equality $\ord(M)=\ord^*(M)$ is valid for all $M \in \PG$. The converse inclusion is not true (see Proposition~\ref{prop:Logarithmic}). In what follows we show some sufficient conditions for the mean to be of generalized power growth.

\begin{lem}\label{lem:ord*char}
Let $M \in \HS$. If  there exist $C \in [0,1]$ such that
\Eq{E:GPGC}{
\liminf_{x \to 0^+}\frac{M(x,1)}{x^\alpha}&=\infty\phantom{0}\qquad \text{ for all } \alpha>C,\\
\limsup_{x \to 0^+}\frac{M(x,1)}{x^\alpha}&=0\phantom{\infty}\qquad \text{ for all } \alpha<C.
}
then $M$ is of generalized power growth (that is $M \in \GPG$) and $\ord^*(M)=C$. 

Conversely, for every $M \in \GPG$ equalities \eq{E:GPGC} are valid with $C=\ord^*(M)$.
\end{lem}
\begin{proof}
First assume that \eq{E:GPGC} is valid. Then by the definition, we have $\uo(M)\le C$ and $\lo(M) \ge C$. Thus, by Lemma~\ref{lem:louocomp}, we obtain $C \le \lo(M) \le \uo(M) \le C$ which implies $\lo(M)=\uo(M)=C$. Therefore $M$ is generalized power growth and $\ord^*(M)=\lo(M)=C$.
The converse implication is easily implied by Lemma~\ref{lem:louocomp}.
\end{proof}

This lemma has the following easy-to-see corollary.
\begin{cor}\label{cor:coincide}
For every $M \in \PG$  we have $M \in \GPG$ and $\ord(M)=\ord^*(M)$. 
\end{cor}

Indeed, we can simply apply the above lemma with $C:=\ord(M)$.

\subsection{Generalizations of Corollary~\ref{cor:PG}}
At the moment we generalize Corollary~\ref{cor:PG} in two ways. First, we prove that this statement can be split into two inequalities involving the lower and the upper order of means. Later we show that Corollary~\ref{cor:PG} remains valid for means which are of generalized power growth.

\vskip4mm
\begin{thm}\label{thm:main}
    Let $M,N,K \in \HS$ such that the ratio $x \mapsto \frac{M(x,1)}{N(x,1)}$ is bounded from above. Define
    \Eq{*}{
    K_{M,N}(x,y):=K(M(x,y),N(x,y)) \qquad (x,y \in \R_+).
    }
Then $K_{M,N} \in \HS$ and
    \Eq{Emain}{
\lo(K_{M,N}) &\ge \big(\lo (M)- \uo(N)\big)\lo(K)+\lo(N); \\
\uo(K_{M,N}) &\le \big(\uo(M)-\lo(N)\big)\uo(K)+\uo(N).
    }
\end{thm}

\begin{proof}
Obviously $K_{M,N} \in \HS$. 
For the sake of brevity, define 
\Eq{*}{
l_K:=\lo(K),&\quad u_K:=\uo(K), \\ 
l_M:=\lo(M),&\quad u_M:=\uo(M), \\
l_N:=\lo(N),&\quad u_N:=\uo(N). 
}
Moreover set $m,n,k,r\colon \R_+\to \R_+$ by
\Eq{*}{
m(x):=M(x,1), \quad n(x):=N(x,1), \quad k(x):=K(x,1), \quad r(x):=K_{M,N}(x,1) \quad (x \in \R_+).
}

Then
\Eq{*}{
r(x)=K_{M,N}(x,1)=K\big(M(x,1),N(x,1)\big)=K\big(m(x),n(x)\big),\qquad \text{ for all }x \in \R_+.
}
Since $K \in \HS$ we get
\Eq{*}{
r(x)=n(x) k\big(\tfrac{m(x)}{n(x)}\big)=m(x) k\big(\tfrac{n(x)}{m(x)}\big),\qquad \text{ for all } x \in \R_+.
}

By the assumption (since both $M$ and $N$ are homogeneous) we have 
\Eq{*}{
c:=\sup\bigg\{\frac{m(x)}{n(x)}\colon x \in (0,\infty)\bigg\}=\sup\bigg\{\frac{M(x,1)}{N(x,1)}\colon x \in (0,\infty)\bigg\}\in [1,\infty).}
Thus, by Proposition~\ref{Prop33},  $l_M \ge l_N$ and $u_M \ge u_N$. Moreover, for all $x \in \R_+$,
\Eq{*}{
r(x)=n(x) k\big(\tfrac{m(x)}{n(x)}\big)
 \le cn(x),\quad\text{ and }\quad  
r(x)= m(x) k\big(\tfrac{n(x)}{m(x)}\big) \ge\tfrac1c  m(x),
}
which, applying Proposition~\ref{Prop33} again, implies 
\Eq{*}{
l_N \le \lo(K_{M,N})\le l_M;\quad
u_N\le \uo(K_{M,N})\le u_M.
}

Using these inequalities one can show the first part of \eq{Emain} when $l_K=0$. Therefore, we can restrict the proof of the first inequality to the case $l_K>0$.

\case{1}
If $\limsup_{x\to 0^+}\tfrac{m(x)}{n(x)}>0$ then take a sequence $(x_n)$ converging to $0$ such that $\big(\tfrac{m(x_n)}{n(x_n)}\big)_{n=1}^\infty$ is convergent and has a positive limit. Then $\delta:=\inf\big\{\tfrac{m(x_n)}{n(x_n)}\colon n \in \N\big\}>0$. Thus
\Eq{*}{
cn(x_n) \ge r(x_n)=K(m(x_n),n(x_n))\ge \min(m(x_n),n(x_n)) \ge \delta n(x_n).
}
Whence
\Eq{*}{
\limsup_{x \to 0^+} \frac{r(x)}{x^\alpha} \ge \delta \liminf_{x \to 0^+} \frac{n(x)}{x^\alpha}=+\infty \text{ for all }\alpha>u_N,
}
which yields $\uo(K_{M,N}) \le u_N$.  Thus $u_M=u_N=\uo(K_{M,N})$, which implies the second inequality of \eq{Emain}. To show the first inequality note that
\Eq{*}{
(l_M- u_N)l_K+l_N\le (u_M-u_N)l_K+l_N=l_N\le \lo(K_{M,N}).
}

\case{2} If $\lim_{x \to 0^+} \frac{m(x)}{n(x)}=0$ then, by Lemma~\ref{lem:louocomp} for all $\varepsilon \in (0,l_K)$, we have 
\Eq{*}{
\lim_{x \to 0^+} r(x)x^{\varepsilon-l_N+(l_K-\varepsilon)(u_N-l_M+2\varepsilon)}=
\lim_{x \to 0^+} \frac{n(x)}{x^{l_N-\varepsilon}} \frac{k\big(\tfrac{m(x)}{n(x)}\big)}{\big(\tfrac{m(x)}{n(x)}\big)^{l_K-\varepsilon}} \Big(\frac{m(x)}{x^{l_M-\varepsilon}}\frac{x^{u_N+\varepsilon}}{n(x)}\Big)^{l_K-\varepsilon}=0.
}

Therefore, applying definition of $\lo$ directly, we get
\Eq{*}{
\lo(K_{M,N})\ge \lim_{\varepsilon\to 0} l_N-\varepsilon-(l_K-\varepsilon)(u_N-l_M+2\varepsilon)=l_K(l_M-u_N)+l_N.
}

Now we proceed to the proof of the second inequality. Analogously, for all $\varepsilon \in(0,\infty)$, we have
\Eq{*}{
\frac{r(x)}{x^{(u_N+\varepsilon)+(u_K+\varepsilon)(u_M-l_N)}}=
\frac{k \Big(\frac{m(x)}{n(x)}\Big)}{\Big(\frac{m(x)}{n(x)}\Big)^{u_K+\varepsilon}} \Big(\frac{m(x)}{x^{u_M+\varepsilon}}\frac{x^{l_N-\varepsilon}}{n(x)}\Big)^{u_K+\varepsilon}\frac{n(x)}{x^{u_N+\varepsilon}}.
}
Therefore 
 \Eq{*}{
 \lim_{x \to 0^+}\frac{r(x)}{x^{(u_M-l_N)u_K+u_N+\varepsilon(1+u_M-l_N)}}=\infty\qquad \text{ for all }\varepsilon \in (0,\infty).
 }
Thus, analogously to the previous case,
\Eq{*}{
\uo(K_{M,N}) \le \inf_{\varepsilon\in (0,\infty)} 
(u_M-l_N)u_K+u_N+\varepsilon(1+u_M-l_N)=(u_M-l_N)u_K+u_N,
}
which concludes the proof of the second inequality.
\end{proof}

 \begin{cor}\label{cor:2.8}
    Let $M,N \in \HS$ be such that the ratio $x \mapsto \frac{M(x,1)}{N(x,1)}$ is bounded from above and $K \colon \R_+^2 \to \R_+$ be the $(M,N)$-invariant mean. Then 
    \begin{enumerate}[(i)]
        \item $\lo(K) \ge \frac{\lo(N)}{1+\uo(N)-\lo(M)}$;
        \item $\uo(K) \le \frac{\uo(N)}{1+\lo(N)-\uo(M)}$ unless $\lo(N)=0$ and $\uo(M)=1$.
    \end{enumerate}
 \end{cor}
\begin{proof}
By the invariance property, we have $K_{M, N}=K$. Then the inequalities in Theorem~\ref{thm:main} simplifies to 
    \Eq{*}{
\lo(K) &\ge \big(\lo (M)- \uo(N)\big)\lo(K)+\lo(N); \\
\uo(K) &\le (\uo(M)-\lo(N))\uo(K)+\uo(N).
    }
Thus to complete the proof, it is sufficient to show the inequalities 
\Eq{*}{
1+\uo(N)-\lo(M)&>0; \\
1+\lo(N)-\uo(M) &> 0 \text{ unless }\lo(N)=0\text{ and }\uo(M)=1.
}
However, analogously to the previous statement, we have $\lo(M) \le \lo(N)$ and $\uo(M) \le \uo(N)$. Thus 
\Eq{*}{
1+\uo(N)-\lo(M) \ge 1+\uo(M)-\lo(M) \ge 1>0.}

To show the second inequality, note that since $(\lo(N),\uo(M))\ne(0,1)$ we have $\uo(M)-\lo(N)<1$, and it easily follows.
\end{proof}

Now we apply this corollary to the case when $M, N \in \GPG$.

\begin{cor}
    Let $M,N \in \GPG$ be such that $\ord^*(M)\ge \ord^*(N)$ and $(\ord^*(M),\ord^*(N))\ne(1,0)$. Moreover, let $K \colon \R_+^2 \to \R_+$ be the $(M,N)$-invariant mean.
    
    Then $K \in \GPG$ and
  \Eq{E:407}{
  \ord^*(K)= \frac{\ord^*(N)}{1 - \ord^*(M) + \ord^*(N)}.
  }
\end{cor}
\begin{proof}
    If the ratio $x \mapsto \frac{M(x,1)}{N(x,1)}$ is bounded from above, then it is a straightforward application of Corollary~\ref{cor:2.8}. 
    
    If this is not the case, then there exists a divergent sequence $(x_k)$ that converges to zero such that $M(x_k,1)\ge N(x_k,1)$. Then, for all $\alpha>\ord^*(N)$, we have
    \Eq{*}{
\infty=\liminf_{x \to 0^+}\frac{N(x,1)}{x^\alpha}\le \liminf_{k \to \infty}\frac{N(x_k,1)}{x_k^\alpha}\le \liminf_{k \to \infty}\frac{M(x_k,1)}{x_k^\alpha} \le \lim_{x \to \infty}\frac{M(x,1)}{x^\alpha} 
}
which, by Lemma~\ref{lem:ord*char}, implies that the inequality $\ord^*(M)\le \alpha$ is also valid. Consequently $\ord^*(M)\le \ord^*(N)$.
    However, by the assumptions, the reverse inequality is also valid. Thus $\ord^*(M)=\ord^*(N)$. Then, since $K$ is $(M,N)$-invariant, we have $\min(M,N)\le K\le \max(M,N)$. Consequently $K \in \GPG$ and $\ord^*(M)=\ord^*(N)=\ord^*(K)$. Then the assertion trivially follows.
\end{proof}
\section{Examples}

We start this section with two simple (although quite artificial) examples which show how to calculate the order. Later we deal with logarithmic mean and the family of Gini means.

 In the first example, we show that not every mean of generalized power growth is of power growth.
\begin{exa}
For a function $e \colon (0,1) \to \R_+$ such that $x \le e(x)\le 1 $ for all $x\in(0,1)$, we define the mean $\mathscr{H}_e \in \HS$ by
 \Eq{M_e}{
 \mathscr{H}_e(x,y):=\max(x,y) e\Big(\frac{\min(x,y)}{\max(x,y)}\Big).
 }

 If $\inf e>0$ then $\mathscr{H}_e$ is of generalized power growth and $\ord^*(\mathscr{H}_e)=0$. 
In the latter case, $\mathscr{H}_e$ is of power growth only if the limit $e(0^+)$ exists.
\end{exa}

\begin{exa}
 We set functions $e_1,e_2 \colon (0,1) \to (0,1)$ by 
 \Eq{*}{
 e_1(t)&:=t^{\frac12+\frac12\sin(\frac1t)} &\qquad t \in(0,1),\\
 e_2(t)&:=\tfrac12 (\sqrt{t}-t)(1+\sin(\tfrac{\pi}{2t}))+t &\qquad t \in(0,1).
 }
Then we define means $\mathscr{H}_{e_1}$ and $\mathscr{H}_{e_2}$ by \eq{M_e}. For every $\alpha \in(0,1)$ we have 
\Eq{*}{
\liminf_{x \to 0^+}\frac{\mathscr{H}_{e_1}(x,1)}{x^\alpha}&=\liminf_{x \to 0^+}\frac{e_1(x)}{x^\alpha}=0;\\ 
\limsup_{x \to 0^+}\frac{\mathscr{H}_{e_1}(x,1)}{x^\alpha}&=\limsup_{x \to 0^+}\frac{e_1(x)}{x^\alpha}=\infty.
}
Therefore $\lo(\mathscr{H}_{e_1})=0$ and $\uo(\mathscr{H}_{e_1})=1$. Similarly $\lo(\mathscr{H}_{e_2})=\tfrac12$ and $\uo(\mathscr{H}_{e_2})=1$.
 \end{exa}

\subsection{Logarithmic mean}
Now we proceed to the logarithmic mean (see, for example, Carlson~\cite{Car72a}). Define the \emph{logarithmic mean} $\calL \colon \R_+^2 \to \R_+$ by 
\Eq{*}{
\calL(x,y)=\frac{x-y}{\log x -\log y}.
}
Obviously $\calL \in \HS$. Moreover we can show the following result
\begin{prop}\label{prop:Logarithmic}
    $\calL$ is not of power growth but is of generalized power growth, and $$\ord^*(\calL)=0.$$
\end{prop}
\begin{proof}
    For all $\alpha \in (-\infty,0)$ we have
\Eq{*}{
\limsup_{x \to 0^+}\frac{\calL(x,1)}{x^\alpha}=\lim_{x \to 0^+} \frac{x-1}{x^\alpha \log x} =
\lim_{x \to 0^+} \frac{\alpha(x-1)}{x^\alpha \log (x^\alpha)}=-\alpha \lim_{y \to \infty}\frac{1}{y\log y}=0. 
}
Similarly, for all $\alpha \in (0,\infty)$ we have
\Eq{*}{
\liminf_{x \to 0^+}\frac{\calL(x,1)}{x^\alpha}=\lim_{x \to 0^+} \frac{x-1}{x^\alpha \log x} =
\lim_{x \to 0^+} \frac{\alpha(x-1)}{x^\alpha \log (x^\alpha)}=-\alpha \lim_{y \to 0^+}\frac{1}{y\log y}=\infty. 
}
Therefore 
\Eq{*}{
\lo(\calL)&=\sup\Big\{\alpha \in \R \colon \limsup_{x \to 0^+}\frac{\calL(x,1)}{x^\alpha}=0 \Big\}\ge 0; \\
\uo(\calL)&=\inf\Big\{\alpha \in \R \colon \liminf_{x \to 0^+}\frac{M(x,1)}{x^\alpha}=\infty \Big\}\le 0.
}
Thus, by Lemma~\ref{lem:louocomp}, have 
$0\le \lo(\calL)\le \uo(\calL) \le 0$,
which yields $\lo(\calL)= \uo(\calL)= 0$. Finally, we obtain that $\calL$ is of generalized power growth and $\ord^*(\calL)=0$.
\end{proof}

\subsection{Gini means} Finally we proceed to the means which were introduced by Gini \cite{Gin38}. For $p,q \in \R$ define the \emph{Gini mean} $\G_{p,q}\colon \R_+^2 \to \R$ by
\Eq{*}{
\G_{p,q}(x,y):=
\begin{cases}
\left(\dfrac{x^p+y^p}{x^q+y^q}\right)^{\frac1{p-q}} & \textrm{if\ }p \ne q\,,\\
\exp \left(\dfrac{x^p\ln x+y^p\ln y}{x^p+y^p}\right)& \textrm{if\ }p = q\,.
\end{cases}
}
Obviously these means are homogeneous and $\G_{p,q}=\G_{q,p}$ for all $p,q \in \R$. Moreover, this family is a generalization of power means since $\G_{p,0}$ is a $p$-th power mean ($p \in \R$). The classical comparability result by D\'aroczy--Losonczi~\cite{DarLos70} states that 
\Eq{*}{
\min(p,q) \le \min(r,s) \text{ and } \max(p,q) \le \max(r,s) \Rightarrow \G_{p,q} \le \G_{r,s},
}
where $p,q,r,s \in\R$. In the case of bivariate means the full characterization of comparability was obtained by P\'ales~\cite{Pal88c}.

\begin{prop}[\!\!\cite{Pal88c}, Theorem 3]
Let $p, q, r, s \in \R$. Then the following conditions are equivalent:
\begin{itemize}
\item For all $x, y > 0$, $\G_{p,q}(x,y)\le \G_{r,s}(x,y)$;
\item $p + q \le r + s$, $m(p, q) \le m(r , s)$, and $\mu(p, q) \le  \mu(r , s)$, where
\Eq{*}{
m(p,q):=
\begin{cases}
\min(p,q) & \text{ if }p,q\ge 0,\\
0& \text{ if } pq<0,\\
\max(p,q) & \text{ if }p,q\le 0;
\end{cases}
\quad 
\mu(p,q):=
\begin{cases}
\frac{\abs{p}-\abs{q}}{p-q} & \text{ if }p\ne q,\\
\sign(p) & \text{ if }p=q;
\end{cases}
}
%\item $(p,q,r,s)\in \Delta_2$, where
%\Eq{*}{
%\Delta_2:=\{(p,q,r,s)\in \R^4&\colon p + q \le r + s,\ m(p, q) \le m(r , s),\\
%&\qquad \mu(p, q) \le  \mu(r , s)\}.
%}
\end{itemize}
\end{prop}

We prove that all Gini means are of generalized power growth and establish their order. 

\begin{prop}
    For all $p,q \in \R$ we have $\G_{p,q} \in \GPG$ and 
    \Eq{*}{
    \ord^*(\G_{p,q})=\begin{cases}
        0 & \text{ if }\min(p,q)\ge 0\text{ and }\max(p,q)>0;\\
        1 & \text{ if }\min(p,q)<0\text{ and }\max(p,q)\le 0;\\
        \frac12 & \text{ if }p=q=0;\\
        \frac{-\min(p,q)}{\abs{p-q}} &\text{ if }\min(p,q)< 0 < \max(p,q).
    \end{cases}
    }
\end{prop}
\begin{proof}
Keeping in mind the equality $\G_{p,q}=\G_{q,p}$, let us assume without loss of generality that $p \ge q$. Now we split our proof into four cases.

\case{1} If $q<0<p$ then we have
    \Eq{*}{
\lim_{x \to 0^+} \frac{\G_{p,q}(x,1)}{x^\alpha}=\lim_{x \to 0^+} 
x^{-\alpha} \bigg(\frac{x^p+1}{(x^{-q}+1)x^q}\bigg)^{1/(p-q)} =
\lim_{x \to 0^+} 
x^{-\alpha-\frac{q}{p-q}} \bigg(\frac{x^p+1}{x^{-q}+1}\bigg)^{1/(p-q)}
}
Since $\lim\limits_{x\to 0^+} \frac{x^p+1}{x^{-q}+1}=1$ we get that $\ord^*(\G_{p,q})$ is the solution $\alpha$ of the equation $-\alpha-\frac{q}{p-q}$, that is $\alpha=\frac{-q}{p-q}=\frac{-\min(p,q)}{\abs{p-q}}$. 

\case{2} If $q \ge 0$ and $p>0$ then 
\Eq{*}{
\lim_{x \to 0^+} \G_{p,q}(x,1)\ge \lim_{x \to 0^+} \G_{0,p}(x,1)=2^{-1/p}>0
}
which (by Lemma~\ref{lem:ord*char}) implies $\ord^*(M)=0$.

\case{3} If $q<0$ and $p \le 0$ then 
\Eq{*}{
\lim_{x \to 0^+} \frac{1}{x}\G_{p,q}(x,1)\le \lim_{x \to 0^+} \frac{1}{x}\G_{0,q}(x,1) =\lim_{x \to 0^+}\G_{0,q}\Big(1,\frac1x\Big)=\lim_{x \to 0^+}\bigg(\frac{1+x^{-q}}2\bigg)^{1/q}=2^{-1/q},
}
which (by Lemma~\ref{lem:ord*char}) implies $\ord^*(M)=1$.

\case{4} At the very end we consider the isolated case $p=q=0$. Then $\G_{p,q}(x,1)=\sqrt{x}$. Consequently \eq{E:GPGC} holds with $C=\tfrac12$. Thus, by Lemma~\ref{lem:ord*char}, we obtain $\ord^*(\G_{0,0})=\tfrac12$. 
\end{proof}

\end{document}